\documentclass[a4paper, fleqn]{article}

\usepackage{amsfonts}
\usepackage{bbm}
\usepackage{amsmath}
\usepackage{amssymb}
\usepackage[english]{babel}
\usepackage{amsthm}
\usepackage{tikz}
\usepackage{textgreek}
\usepackage{mathtools}
\usetikzlibrary{arrows.meta, calc, patterns}
\newtheorem {thm}{Theorem}
\newtheorem {lem}[thm]{Lemma}
\newtheorem {cor}[thm]{Corollary}
\newtheorem {conj}[thm]{Conjecture}
\newtheorem {prop}[thm]{Proposition}
\newtheorem* {que*}{Question}

\theoremstyle{remark}
\newtheorem*{rem*}{Remark}

\DeclareMathOperator{\dist}{dist}

\DeclareMathOperator{\conv}{conv}

\title{On a conjecture by Eckhoff and Dolnikov concerning line transversals to Euclidean disks}
\author{Alexander Magazinov\thanks{Supported in part by ERC Starting Grant 678520.}}

\begin{document}

\maketitle

\begin{abstract}
Let $K$ be a convex body in the Euclidean plane $\mathbb R^2$. We say that
a point set $X \subseteq \mathbb R^2$ satsfies the property $T(K)$ if the family of translates
$\{ K + x : x \in X \}$ has a line transversal. A weaker property, $T(K, s)$, of
the set $X$ is that every subset $Y \subseteq X$ consisting of at most $s$ elements satisfies
the property $T(K)$.

The following question goes back to Gr\"unbaum: given $K$ and $s$, what is the minimal positive
number $\lambda = \lambda(K, s)$ such that every finite point set in $\mathbb R^2$ with the property $T(K, s)$
also satisfies the property $T(\lambda K)$? The constant $\lambda_{disj}(K, s)$ is defined
similarly, with the only additional assumption that the translates $x + K$ and $y + K$
are disjoint for every $x, y \in X$, $x \neq y$.

One case of particular interest is $s = 3$ and $K = B$, where $B$ is a unit Euclidean ball.
Namely, it was conjectured by Eckhoff and, independently, Dolnikov that $\lambda (B, 3) = \frac{1 + \sqrt{5}}{2}$.
 
In this paper we propose a stronger conjecture, which,
on the other hand, admits an algebraic formulation in a finite alphabet. We verify our conjecture
numerically on a sufficiently dense grid in the space of parameters and thereby obtain an estimate
$\lambda_{disj}(B, 3) \leq \lambda(B, 3) \leq 1.645$. This is an improvement on the previously known
upper bounds $\lambda(B, 3) \leq \frac{1 + \sqrt{1 + 4\sqrt{2}}}{2} \approx 1.79$ (Jer\'onimo Castro and Rold\'an-Pensado, 2011) and 
$\lambda_{disj}(B, 3) \leq 1.65$ (Heppes, 2005).
\end{abstract}

\section{Introduction}

Let $K$ be a convex body in the Euclidean plane $\mathbb R^2$. We will consider families
\begin{equation}\label{eq:translates}
  \mathcal F = \{ K + x : x \in X \} \qquad (X \subseteq \mathbb R^2)
\end{equation}
of translates of $K$ with the following property: every $s$-tuple of translates from $\mathcal F$ has a line transversal.
(Here $s > 2$ is a fixed integer number.) The following question has been proposed by Gr\"unbaum~\cite{Gru58}: given $K$
and $s$, what is the minimum value of a constant $\lambda = \lambda(K, s)$ such that for every finite family $\mathcal F$ satisfying
the above property one can guarantee that the family of blow-ups
\begin{equation*}
  \mathcal \lambda F = \{ \lambda K + x : x \in X \}
\end{equation*}
can be stabbed by a single line?

Let us introduce a convenient notation. We will rather consider point sets $X \subseteq \mathbb R^2$ instead
of the associated families of translates~\eqref{eq:translates}. If a family~\eqref{eq:translates} can be stabbed
by a single line, we say that the associated point set $X$ {\it satisfies the property $T(K)$}. If every subset of $X$,
consisting of at most $s$ elements, satisfies the property $T(K)$, then we say that $X$ {\it satisfies the property $T(K, s)$}.
Then Gr\"unbaum's question can be formulated as follows.

\begin{que*}
Given a convex body $K \subset \mathbb R^2$ and an integer $s > 2$, what is the minimum possible value $\lambda(K, s)$ such that
the following holds: every finite point set $X \subset \mathbb R^2$ with the property $T(K, s)$
necessarily satisfies the property $T(\lambda K)$?
\end{que*}

Applying additional restriction to the point set $X$ may also make sense. For instance, let us call $X$
{\it $K$-separated} if the family~\eqref{eq:translates} consists of pairwise disjoint translates. Then one can
ask an analogous question.

\begin{que*}
Given a convex body $K \subset \mathbb R^2$ and an integer $s > 2$, what is the minimum possible value $\lambda_{disj}(K, s)$ such that
the following holds: every $K$-separated finite point set $X \subset \mathbb R^2$ with the property $T(K, s)$
necessarily satisfies the property $T(\lambda K)$?
\end{que*}

The case $K = B$, $s = 3$, where
\begin{equation*}
  B = \{ (x, y) : x^2 + y^2 \leq 1 \},
\end{equation*}
i.e., $B$ is the unit ball centered at the origin, has attracted some particular attention. The following conjecture
was posed in 1969 by Eckhoff~\cite{Eck69} and, independently, in 1972 by Dolnikov (see~\cite{Jer07}).

\begin{conj}[Dolnikov, Eckhoff]\label{conj:DE}
$\lambda (B, 3) = \frac{1 + \sqrt{5}}{2}$.
\end{conj}

Conjecture~\ref{conj:DE} remains unresolved so far (see~\cite{Eck16}). Recently, it was highlighted in the Handbook of Discrete and Convex Geometry
(see~\cite[Conjecture~4.2.25]{Hdb17}) as an important problem in the theory of geometric transversals.

Some partial results towards Conjecture~\ref{conj:DE} have been achieved.
It is known that $\lambda(B, 3) \geq \lambda_{disj}(B, 3) \geq \frac{1 + \sqrt{5}}{2}$, as implied by considering
the vertex set of a regular pentagon with each side equal to $\frac{2}{\sqrt{\tau + 2}}$, where $\tau = \frac{1 + \sqrt{5}}{2}$.
The best previously known upper bounds are $\lambda(B, 3) \leq \frac{1 + \sqrt{1 + 4\sqrt{2}}}{2} \approx 1.79$ due to Jer\'onimo Castro and Rold\'an-Pensado~\cite{JR11},
and $\lambda_{disj}(B, 3) \leq 1.65$ due to Heppes~\cite{Hep05}.

In this paper we improve the known upper bounds for both $\lambda(B, 3)$ and $\lambda_{disj}(B, 3)$. Moreover, we provide some significant evidence
indicating that our approach can resolve Conjecture~\ref{conj:DE} completely.

\section{The ``finite'' conjecture, the parametrization and the restriction to a grid}

Let us pose a conjecture, which is, apparently, stronger than Conjecture~\ref{conj:DE}.

\begin{conj}\label{conj:finite}
Let $E \subset \mathbb R^2$ be an elliptical disk, $Z \subset \partial E$ be a point set
such that $\# Z \leq 5$ and $E$ has the minimum area of all elliptic disks containing $E$.
Consider the set
\begin{equation}\label{eq:r}
R = R(Z) = \{ x \in \mathbb R^2 : \text{ $\{ x \} \cup Z$ satisfies the property $T(B, 3)$} \}.
\end{equation}
Then the set $R \cap E$ satisfies the property $T \left( \frac{1 + \sqrt{5}}{2} B \right)$.
\end{conj}

We will use the following parametrization of Conjecture~\ref{conj:finite}. Consider the cartesian coordinate system $(x, y)$,
where
\begin{equation}\label{eq:param_1}
  E = \left\{ (x, y) \in \mathbb R^2 : \frac{x^2}{r_1^2} + \frac{y^2}{r_2^2} \leq 1 \right\}, \quad r_1 \geq r_2.
\end{equation}
If $Z = \{ z_1, z_2, \ldots, z_k \}$, $k \leq 5$, then
\begin{equation}\label{eq:param_2}
  z_i = (r_1 \cos \alpha_i, r_2 \sin \alpha_i), \quad \text{for} \quad i = 1, 2, \ldots, k.
\end{equation}
Without loss of generality one can assume
\begin{equation}\label{eq:param_3}
0 \leq \alpha_1 < \alpha_2 < \ldots < \alpha_k < 2\pi.
\end{equation}
Thus Conjecture~\ref{conj:finite} gets parameterized by $r_1, r_2, \alpha_1, \alpha_2, \ldots, \alpha_k$.

\begin{rem*}
One can eliminate the trigonometric expressions in the parametrization by the standard substitution $t_i = \tan \frac{\alpha_i}{2}$.
The condition that $E$ is the ellipsoid of minimal area containing $Z$ is algebraic in $t_i$ (see, for example,~\cite{Grr11}). The set $R$ is
defined algebraically in $r_1$, $r_2$ and $t_i$. Therefore Conjecture~\ref{conj:finite} has algebraic parametrization with
at most 7 variables.
\end{rem*}

We are ready to state the main results of this paper.

\begin{thm}\label{thm:reduction}
If Conjecture~\ref{conj:finite} holds, then Conjecture~\ref{conj:DE} holds, too.
\end{thm}

\begin{thm}\label{thm:grid}
Let $(k; r_1, r_2; \alpha_1, \alpha_2, \ldots, \alpha_k)$ be the parameters as in~\eqref{eq:param_1},~\eqref{eq:param_2} and~\eqref{eq:param_3}. Then
Conjecture~\ref{conj:finite} is true in the following cases:
\begin{itemize}
  \item[(a)] $k = 3$.
  \item[(b)] $k = 4$, $r_1, r_2 \in 0.015 \mathbb Z$.
  \item[(c)] $k = 5$, $r_1, r_2 \in 0.015 \mathbb Z$, and $\alpha_i \in \frac{\pi}{960} (\mathbb Z + 1/2)$ for each $i = 1, 2, \ldots, 5$.
\end{itemize}
\end{thm}

\begin{thm}\label{thm:estimate}
$\lambda_{disj}(B, 3) \leq \lambda (B, 3) \leq  1.645$.
\end{thm}

We conclude this section with a brief guide over the contents of the rest of the paper.

\begin{itemize}
  \item Section~\ref{sec:thm_1} contains the proof of Theorem~\ref{thm:reduction}.
  \item Section~\ref{sec:aux} accommodates a number of auxiliary statements necessary for the further argument. We formulate
        some useful corollaries of the so-called John representation associated with the minimum area elliptic disk. Then we provide
        several simple tools for extrapolating estimates on a finite subset of the parameter space to the subspace covered by Theorem~\ref{thm:grid},
        and then to the entire parameter space. Finally, we state two lemmas that allow us eliminate all ``too long'' elliptical disks from consideration.
        The proofs are given in the subsequent sections.
  \item Sections~\ref{sec:john_proofs}--\ref{sec:lem_2_proof} are devoted to the proofs of statements formulated in Section~\ref{sec:aux}. The only proof
        that remains postponed is the one of Lemma~\ref{lem:1.5}.
  \item Section~\ref{sec:computer_proofs} addresses the computer-assisted proofs, namely the ones of Lemma~\ref{lem:1.5} and parts (b)--(c) of Theorem~\ref{thm:grid}.
  \item Section~\ref{sec:grid_a} contains a short (non-computer-assisted) proof of part (a) of Theorem~\ref{thm:grid}.
  \item Section~\ref{sec:proof_est} reduces Theorem~\ref{thm:estimate} to Theorem~\ref{thm:grid}, which is proved earlier.
\end{itemize}

\section{Reduction of the Dolnikov--Eckhoff conjecture to Conjecture~\ref{conj:finite}}\label{sec:thm_1}

Let us show that Conjecture~\ref{conj:finite} indeed implies Conjecture~\ref{conj:DE}.

\begin{proof}[Proof of Theorem~\ref{thm:reduction}]
Let $X$ be a counterexample to Conjecture~\ref{conj:DE}. One can choose a sufficiently small constant
$\varepsilon > 0$ so that every sufficiently small perturbation $X'$ of the set $X$ is still a counterexample to Conjecture~\ref{conj:DE}.
Let us choose $X'$ to be sufficiently generic so there is no ellipse passing through 6 or more points of $X'$.

Let $E$ be an elliptical disk of minimal area containing the set $X'$. Denote $Z = X' \cap \partial E$.
Since $X'$ is generic, we have $\# Z \leq 5$. In addition, by [ref:Ball], $E$ is the (only) elliptical disk of minimal
area containing $Z$.

Let $R$ be defined according to~\eqref{eq:r}. Then, since $X'$ satisfies the property $T(B, 3)$, we have
\begin{equation*}
  X' \subset E \cap R.
\end{equation*}
But if Conjecture~\ref{conj:finite} holds, then $X'$ satisfies the property $T \left( \frac{1 + \sqrt{5}}{2} B \right)$.
This contradicts our previous assumption that $X'$ is a counterexample to Conjecture~\ref{conj:DE}.
\end{proof}

\section{Some auxiliary results}\label{sec:aux}

\subsection{John--Ball criterion of the minimum area elliptic disk}

This subsection is based on the following Proposition~\ref{prop:john_ball}. The
proof, in any dimension, not only in the plane, can be found in~\cite{Joh48}
(the necessary property of the minimum area ellipsoid) and in~\cite{Bal92}
(the sufficiency). See also~\cite{Grr11} for, perhaps, a more accessible exposition.

\begin{prop}\label{prop:john_ball}
Let $X \subset \mathbb R^2$ be a finite set. Then
\begin{enumerate}
  \item There exists an elliptical disk $E \supset X$ such that every
  elliptical disk $E' \supset X$ that is distinct from $E$ satisfies
  $|E'| > |E|$. (I.e., the minimum area elliptical disk containing $X$
  is unique.)
  \item The following assertions are equivalent.
  \begin{itemize}
    \item[(i)] $B$ is the minimum area elliptic disk containing $X$.
    \item[(ii)] $X \subset B$ and there exists a subset
    $\{x_1, x_2, \ldots, x_k\} \subset X \cap \partial B$ and $k$
    positive numbers $c_1, c_2, \ldots, c_k$ such that
\begin{equation}\label{eq:john_ball}
  \begin{aligned}
    c_1 x_1 + c_2 x_2 + \ldots + c_k x_k & = \mathbf 0, \\
    c_1 x_1 \otimes x_1 + c_2 x_2 \otimes x_2 + \ldots +
    c_k x_k \otimes x_k & = \operatorname{Id}
  \end{aligned}
\end{equation}
  \end{itemize}
\end{enumerate}
\end{prop}

Proposition~\ref{prop:john_ball} will be used through the three corollaries below. Before we turn to the corollaries, let us introduce a functional playing a crucial role in the subsequent argument. Namely, define
\begin{equation}
F(\alpha_1, \alpha_2, \alpha_3, \alpha_4) =
\cos \frac{1}{2}(\alpha_1 + \alpha_2 - \alpha_3 - \alpha_4) +
\cos \frac{1}{2}(\alpha_1 - \alpha_2 + \alpha_3 - \alpha_4) +
\cos \frac{1}{2}(\alpha_1 - \alpha_2 - \alpha_3 + \alpha_4).
\end{equation}

\begin{cor}\label{cor:simplex}
Let
\begin{equation*}
  X = \{ (\cos \alpha_i, \sin \alpha_i) : i = 1, 2, \ldots, k \}.
\end{equation*}
Then the following assertions are equivalent.
  \begin{itemize}
    \item[(i)] $B$ is the minimum area elliptic disk containing $X$.
    \item[(ii)] $(0, 0, 0, 0) \in \conv \{ (\cos \alpha_i, \sin \alpha_i, \cos 2\alpha_i, \sin 2\alpha_i) : i = 1, 2, \ldots, k \}$.
  \end{itemize}
\end{cor}

\begin{cor}\label{cor:no_large_arcs}
Let
\begin{equation*}
  X = \{ (\cos \alpha_i, \sin \alpha_i) : i = 1, 2, \ldots, k \}.
\end{equation*}
be a finite set such that $B$ is the minimum area elliptical disk containing $X$. Assume, additionally, that
\begin{equation*}
  0 \leq \alpha_1 < \alpha_2 < \ldots < \alpha_k < 2\pi.
\end{equation*}
If $\alpha_{k + 1} = \alpha_1 + 2\pi$, then the inequality $\alpha_{i + 1} - \alpha_i \leq \frac{2\pi}{3}$ holds for
every $i = 1, 2, \ldots, k$.
\end{cor}

\begin{cor}\label{cor:four_point_identity}
Let $x_1, x_2, x_3, x_4 \in \partial B$ be four points such that the
identities~\eqref{eq:john_ball} hold with some positive coefficients $c_1, c_2, c_3, c_4$. Let $x_i = (\cos \alpha_i, \sin \alpha_i)$, where
Then $F(\alpha_1, \alpha_2, \alpha_3, \alpha_4) = 0$. \newline
Equivalently, if
$\alpha_1 < \alpha_2 < \alpha_3 < \alpha_4 < \alpha_1 + 2\pi$ and $\phi_i = \frac{\alpha_{i + 1} - \alpha_i}{2}$
($i = 1, 2, 3, 4$, $\alpha_5 = \alpha_1 + 2\pi$), then
\begin{equation}\label{eq:four_points}
  \cos(\phi_2 - \phi_4) + \cos(\phi_2 + \phi_4) + \cos(2 \phi_1 + \phi_2 + \phi_4) = 0.
\end{equation}
\end{cor}

\begin{rem*}
A cyclic shift of the 4-tuple $(\phi_1, \phi_2, \phi_3, \phi_4)$ turns~\eqref{eq:four_points} into itself.
\end{rem*}

\begin{cor}\label{cor:five_point_ineq}
Let $x_1, x_2, \ldots, x_5 \in \partial B$ be five points such that the
identities~\eqref{eq:john_ball} hold with some positive coefficients $c_1, c_2, \ldots, c_5$. Let $x_i = (\cos \alpha_i, \sin \alpha_i)$, where
\begin{equation*}
\alpha_1 < \alpha_2 < \alpha_3 < \alpha_4 < \alpha_5 < \alpha_1 + 2\pi.
\end{equation*}
Then the following five real numbers:
\begin{equation}\label{eq:5_functionals}
F(\alpha_1, \alpha_2, \alpha_3, \alpha_4), \;
F(\alpha_2, \alpha_3, \alpha_4, \alpha_5), \;
- F(\alpha_3, \alpha_4, \alpha_5, \alpha_1), \;
F(\alpha_4, \alpha_5, \alpha_1, \alpha_2), \;
- F(\alpha_5, \alpha_1, \alpha_2, \alpha_3)
\end{equation}
are either all negative or all positive. Conversely, if $x_i$ and
$\alpha_i$ are as above, and the values~\eqref{eq:5_functionals} are either all negative or all positive, then $B$ is the minumum area elliptical disk containing the set $\{ x_1, x_2, x_3, x_4, x_5 \}$.
\end{cor}

Corollaries~\ref{cor:simplex}--\ref{cor:five_point_ineq} are proved in Section~\ref{sec:john_proofs}.

\subsection{Approximation lemmas}

Since the computer verification is possible only for a finite subset (though a dense one) in the space of parameters, we will need to perform an extrapolation
to the entire parameter space. This subsection provides some simple tools for such an extrapolation.

\begin{lem}\label{lem:contraction}
Let $M : \mathbb R^2 \to \mathbb R^2$ be a non-degenerate affine map such that for every $x, y \in \mathbb R^2$ the inequality
\begin{equation}\label{eq:contraction}
  \| Mx - My \| \leq \| x - y \|
\end{equation}
holds. Assume that a finite set $X$ satisfies the property $T(rB)$ for some $r > 0$. Then the set $MX$ satisfies the property $T(rB)$, too.
\end{lem}

\begin{lem}\label{lem:lipschitz}
Assume that a finite set $X \in \mathbb R^2$ satisfies the property $T(rB)$ for some $r > 0$. Let $\varepsilon > 0$ and let a finite set $Y \subset \mathbb R^2$
satisfy $Y \subset X + \varepsilon B$. Then $Y$ satisfies the property $T((r + \varepsilon)B)$.
\end{lem}

\begin{lem}\label{lem:functional_lipschitz}
Let $\alpha_i, \alpha'_i \in \mathbb R$ ($i = 1, 2, 3, 4$) and $\varepsilon > 0$ be given such that $\max | \alpha_i - \alpha'_i | \leq \varepsilon$. Then
\begin{equation*}
  | F(\alpha_1, \alpha_2, \alpha_3, \alpha_4) - F(\alpha'_1, \alpha'_2, \alpha'_3, \alpha'_4) | < 3 \varepsilon.
\end{equation*}
\end{lem}

The proofs are provided in Section~\ref{sec:perturb}.

\subsection{Some a priori bounds}

The results of this subsection will allow us eliminate all ``too long'' elliptical disks from consideration. In other words,
by using the lemmas below we will restrict ourselves to a compact subset of the parameter space.

\begin{lem}\label{lem:2}
Let $r > 2$. Let $X$ be a finite set such that $rB$ is the minimum area elliptical disk that contains $X$.
Then $X$ violates the property $T(B, 3)$.
\end{lem}

\begin{rem*}
By Lemma~\ref{lem:2} and Lemma~\ref{lem:contraction}, if $X$ satisfies the property $T(B, 3)$ and $E$ is the elliptical disk of minimum volume containing $X$,
then the smaller radius of $E$ does not exceed 2. Therefore $X$ satisfies the property $T(2B)$. This gives an alternative proof for the
Eckhoff's bound $\lambda(B, 3) \leq 2$.
\end{rem*}

Lemma~\ref{lem:2} is proved in Section~\ref{sec:lem_2_proof}.

\begin{lem}\label{lem:1.5}
Consider the elliptical disk
\begin{equation*}
  E = \left\{ (x, y) \in \mathbb R^2 : \frac{x^2}{3^2} + \frac{y^2}{1.62^2} \leq 1 \right\}.
\end{equation*}
Let $X$ be a finite set such that $E$ is the minimum area elliptical disk that contains $X$.
Then $X$ violates the property $T(B, 3)$.
\end{lem}

The proof of Lemma~\ref{lem:1.5} is computer-assisted. As all the other computer-assisted proofs,
it is addressed in Section~\ref{sec:computer_proofs}.

\section{Minimum area elliptic disk: proofs of the key properties}\label{sec:john_proofs}

The aim of this section is to prove Corollaries~\ref{cor:simplex}--\ref{cor:five_point_ineq}.

\begin{proof}[Proof of Corollary~\ref{cor:simplex}]
One can rewrite the identities~\eqref{eq:john_ball} as follows:
\begin{equation}\label{eq:john_ball_expl}
\begin{aligned}
  c_1 \cos \alpha_1 + c_2 \cos \alpha_2 + \ldots + c_k \cos \alpha_k & = 0, \\
  c_1 \sin \alpha_1 + c_2 \sin \alpha_2 + \ldots + c_k \sin \alpha_k & = 0, \\
  c_1 \cos^2 \alpha_1 + c_2 \cos^2 \alpha_2 + \ldots + c_k \cos^2 \alpha_k & = 1, \\
  c_1 \sin^2 \alpha_1 + c_2 \sin^2 \alpha_2 + \ldots + c_k \sin^2 \alpha_k & = 1, \\
  c_1 \sin \alpha_1 \cos \alpha_1 + c_2 \sin \alpha_2 \cos \alpha_2 + \ldots + c_k \sin \alpha_k \cos \alpha_k & = 0.
\end{aligned}
\end{equation}

Taking the difference of the third and the fourth lines of~\eqref{eq:john_ball_expl} yields
\begin{equation*}
  c_1 \cos 2\alpha_1 + c_2 \cos 2\alpha_2 + \ldots + c_k \cos 2\alpha_k = 0.
\end{equation*}
At the same time, multiplying the fifth line of~\eqref{eq:john_ball_expl} by 2 yields
\begin{equation*}
  c_1 \sin 2\alpha_1 + c_2 \sin 2\alpha_2 + \ldots + c_k \sin 2\alpha_k = 0.
\end{equation*}
Aggregately, one concludes that
\begin{equation*}
  \sum\limits_{i = 1}^k c_i (\cos \alpha_i, \sin \alpha_i, \cos 2\alpha_i, \sin 2\alpha_i) = \mathbf 0.
\end{equation*}
This proves the implication $(i) \Rightarrow (ii)$.

Now assume that $(ii)$ holds. Then there are non-negative coefficients $c'_1, c'_2, \ldots, c'_k$, not all of which are zero, such that
\begin{equation*}
  \sum\limits_{i = 1}^k c'_i (\cos \alpha_i, \sin \alpha_i, \cos 2\alpha_i, \sin 2\alpha_i) = \mathbf 0.
\end{equation*}
Then
\begin{gather*}
  c'_1 x_1 + c'_2 x_2 + \ldots + c'_k x_k = \mathbf 0 \quad \text{and} \\
  c'_1 x_1 \otimes x_1 + c'_2 x_2 \otimes x_2 + \ldots + c'_k x_k \otimes x_k = \lambda \operatorname{Id},
\end{gather*}
where $\lambda > 0$. Taking $c_i = \frac{c'_i}{\sqrt{\lambda}}$ proves~\eqref{eq:john_ball} and therefore the implication $(ii) \Rightarrow (i)$.
\end{proof}

\begin{proof}[Proof of Corollary~\ref{cor:no_large_arcs}]
We argue by contradiction. Let the conclusion of Corollary~\ref{cor:no_large_arcs} be false. Then, with no loss of generality, we can assume
that $\alpha_1 = 0$, $\alpha_2 > \frac{2\pi}{3}$.

Then every $\alpha_i$ satisfies the inequality
\begin{equation*}
  \cos(\alpha_i + \pi/3) - \cos(2\alpha_i - \pi/3) \leq 0,
\end{equation*}
and the equality can be achieved only if $\alpha_i = 0$ or $\alpha_i = \frac{4\pi}{3}$. Then all 4-tuples $(\cos \alpha_i, \sin \alpha_i, \cos 2\alpha_i, \sin 2\alpha_i)$
belong to the half-space $\ell(t_1, t_2, t_3, t_4) < 0$, where
\begin{equation*}
  \ell(t_1, t_2, t_3, t_4) = \frac{1}{2} t_1 - \frac{\sqrt{3}}{2} t_2 - \frac{1}{2} t_3 - \frac{\sqrt{3}}{2} t_4.
\end{equation*}
Moreover, if the 4-tuple belongs to the boundary of that half-space, then either $\alpha_i = 0$ or $\alpha_i = \frac{4\pi}{3}$. This contradicts the conclusion of
Corollary~\ref{cor:simplex}.
\end{proof}

\begin{proof}[Proof of Corollary~\ref{cor:four_point_identity}]
With no loss of generality, assume from the very beginning that $\alpha_1 < \alpha_2 < \alpha_3 < \alpha_4 < \alpha_1 + 2\pi$.

Corollary~\ref{cor:simplex} implies that the affine dimension of the 5-tuple of points
\begin{equation*}
  \{ \mathbf 0 \} \cup \{ (\cos \alpha_j, \sin \alpha_j, \cos 2\alpha_j, \sin 2\alpha_j) : j = 1, 2, 3, 4 \}
\end{equation*}
cannot be equal to 4. Therefore
\begin{equation*}
\Delta(\alpha_1, \alpha_2, \alpha_3, \alpha_4) = \left|
  \begin{array}{cccc}
    \cos \alpha_1 & \cos \alpha_2 & \cos \alpha_3 & \cos \alpha_4 \\
    \sin \alpha_1 & \sin \alpha_2 & \sin \alpha_3 & \sin \alpha_4 \\
    \cos 2\alpha_1 & \cos 2\alpha_2 & \cos 2\alpha_3 & \cos 2\alpha_4 \\
    \sin 2\alpha_1 & \sin 2\alpha_2 & \sin 2\alpha_3 & \sin 2\alpha_4 \\
  \end{array}
\right| = 0.
\end{equation*}
One can check that
\begin{multline*}
  \Delta = C \cdot \frac{e_1^2e_2^2 + e_1^2e_3^2 + e_1^2e_4^2 + e_2^2e_3^2 + e_2^2e_4^2 + e_3^2e_4^2}{e_1e_2e_3e_4} \times \\
  \frac{(e_1^2 - e_2^2)(e_1^2 - e_3^2)(e_1^2 - e_4^2)(e_2^2 - e_3^2)(e_2^2 - e_4^2)(e_3^2 - e_4^2)}{(e_1e_2e_3e_4)^3},
\end{multline*}
where $C \in \mathbb R$ is a fixed constant and $e_j = \exp (i \alpha_j / 2)$. But if $j < j'$, then $\alpha_{j'} - \alpha_j \in (0, 2\pi)$. Therefore
\begin{equation*}
  e_j^2 - e_{j'}^2 \in i \mathbb R_{-} \cdot e_j e_{j'}.
\end{equation*}
Consequently, the fraction $\frac{(e_1^2 - e_2^2)(e_1^2 - e_3^2)(e_1^2 - e_4^2)(e_2^2 - e_3^2)(e_2^2 - e_4^2)(e_3^2 - e_4^2)}{(e_1e_2e_3e_4)^3}$
attains only real negative values, hence the sign of $\Delta(\alpha_1, \alpha_2, \alpha_3, \alpha_4)$ is completely determined by the sign of
\begin{equation*}
  \frac{e_1^2e_2^2 + e_1^2e_3^2 + e_1^2e_4^2 + e_2^2e_3^2 + e_2^2e_4^2 + e_3^2e_4^2}{e_1e_2e_3e_4} = 2F(\alpha_1, \alpha_2, \alpha_3, \alpha_4).
\end{equation*}
(Of course, the last expression is a real number.) In particular, if $\Delta(\alpha_1, \alpha_2, \alpha_3, \alpha_4) = 0$, then
$F(\alpha_1, \alpha_2, \alpha_3, \alpha_4) = 0$, as required.
\end{proof}

\begin{proof}[Proof of Corollary~\ref{cor:five_point_ineq}]
We have to check whether the origin $\mathbf 0 \in \mathbb R^4$ belongs to the interior of the simplex
\begin{equation*}
  \conv \{ (\cos \alpha_j, \sin \alpha_j, \cos 2\alpha_j, \sin 2\alpha_j) : j = 1, 2, \ldots, 5 \}.
\end{equation*}
The necessary and sufficient condition is that the determinants
\begin{equation*}
  \Delta(\alpha_1, \alpha_2, \alpha_3, \alpha_4), \Delta(\alpha_5, \alpha_1, \alpha_2, \alpha_3), \ldots, \Delta(\alpha_2, \alpha_3, \alpha_4, \alpha_5)
\end{equation*}
have the same sign. Equivalently, the 5 values~\eqref{eq:5_functionals} have the same sign. (The equivalence is established similarly to the proof
of Corollary~\ref{cor:four_point_identity}).
\end{proof}

\section{Perturbation lemmas: the proofs}\label{sec:perturb}

The goal of this section is to prove Lemmas~\ref{lem:contraction}--\ref{lem:functional_lipschitz}.

\begin{proof}[Proof of Lemma~\ref{lem:contraction}]
Let $l$ be a line such that $\dist (x, l) \leq r$ for every $x \in X$. Such a line exists because $X$ satisfies the property $T(rB)$.
Then $\dist (Mx, Ml) \leq \dist(x, l) \leq r$, therefore the line $Ml$ intersects every translate $rB + y$, where $y$ runs through the set $MX$.
Hence the set $MX$ indeed satisfies the property $T(rB)$.
\end{proof}

\begin{proof}[Proof of Lemma~\ref{lem:lipschitz}]
Let $l$ be a line such that $\dist (x, l) \leq r$ for every $x \in X$. Such a line exists because $X$ satisfies the property $T(rB)$.

Let $y \in Y$. Then there exists a point $x \in X$ such that $\| x - y \| \leq \varepsilon$. Consequently, $\dist(y, l) \leq \dist(x, l) + \varepsilon \leq r + \varepsilon$.
Thus the line $l$ intersects every translate $(r + \varepsilon)B + y$, where $y$ runs through the set $Y$. Therefore
Hence the set $Y$ indeed satisfies the property $T((r + \varepsilon)B)$.
\end{proof}

\begin{proof}[Proof of Lemma~\ref{lem:functional_lipschitz}]
Let $\beta_i = \alpha'_i - \alpha$. Then
\begin{equation*}
\begin{aligned}
  & | F(\alpha_1, \alpha_2, \alpha_3, \alpha_4) - F(\alpha'_1, \alpha'_2, \alpha'_3, \alpha'_4) | \\
  \leq & \frac{1}{2}(|\beta_1 + \beta_2 - \beta_3 - \beta_4| + |\beta_1 - \beta_2 + \beta_3 - \beta_4| + |\beta_1 - \beta_2 - \beta_3 + \beta_4|) \\
  \leq & \frac{1}{2} \cdot 6 \varepsilon = 3\varepsilon,
\end{aligned}
\end{equation*}
as required.

Indeed, the expression in the second line is convex in $(\beta_1, \beta_2, \beta_3, \beta_4)$, therefore it is sufficient to check the inequality
only at the vertices of the cube $[-\varepsilon, \varepsilon]^4$. It is also clear that the inequalities cannot turn into equalities simultaneously.
\end{proof}

\section{The minimum area elliptical disk has width $\leq 2$}\label{sec:lem_2_proof}

In this section we prove Lemma~\ref{lem:2}. Once the proof is complete,
we immediately conclude that Conjecture~\ref{conj:finite} holds whenever $r_1 > 2$. Indeed, in this case we necessarily have $R(X) = \varnothing$ (see the remark after Lemma~\ref{lem:2}).

\begin{proof}[Proof of Lemma~\ref{lem:2}]
With no loss of generality we can assume that $X \subseteq \partial(rB)$ and $\# X \leq 5$. Indeed, if this is not the case, apply a sufficiently small perturbation to $X$, yielding a set
$X'$ with no 6 point lying on the same ellipse. If $E$ is the minimum area elliptical disk containing $X'$, consider the affine map $F$ such that
$F(E) = \frac{2 + r}{2} \cdot B$. Clearly, $F$ is a contraction (if $X$ and $X'$ are sufficiently close to each other). Therefore, by Lemma~\ref{lem:contraction},
it will be sufficiently to argue for $F(X') \cap F(E)$ instead of $X$ and for $\frac{2 + r}{2}$ instead of $r$.

We proceed by case analysis.

\noindent{\bf Case 1.} $\# X = 3$. By condition of the lemma, the ball $rB$ is the minimum area ellipsoid containing the triangle $\conv X$. By a well-known fact [ref:],
this is only possible if $\conv X$ is a regular triangle. But then each height of $\conv X$ equals $\frac{3r}{2} > 3 > 2$, which contradicts the $T(B, 3)$ property.

\noindent{\bf Case 2.} $\# X = 4$. In the notation of Corollary~\ref{cor:four_point_identity}, we can, with no loss of generality, assume that $\phi_1 = \max(\phi_1, \phi_2, \phi_3, \phi_4)$.
In particular, since $\phi_1 + \phi_2 + \phi_3 + \phi_4 = \pi$, we have $\phi_1 \geq \frac{\pi}{4}$

Now we claim that
\begin{equation}\label{eq:max_at_least_quarter_pi}
  \max(\phi_2, \phi_4) \geq \frac{\pi}{4}.
\end{equation}
Indeed, otherwise
\begin{equation*}
  \cos(\phi_2 - \phi_4) + \cos(\phi_2 + \phi_4) = 2 \cos \phi_2 \cos \phi_4 > 2 \cos^2 \frac{\pi}{4} = 1 \geq - \cos(2 \phi_1 + \phi_2 + \phi_4),
\end{equation*}
which contradicts~\eqref{eq:four_points}. The claim~\eqref{eq:max_at_least_quarter_pi} is proved. With no loss of generality, let $\phi_2 \geq \frac{\pi}{4}$.

Finally, using Corollary~\ref{cor:no_large_arcs}, we get
\begin{equation*}
  \frac{2\pi}{3} \geq \phi_1 + \phi_2 = \pi - (\phi_3 + \phi_4) \geq \pi - \frac{2\pi}{3} = \frac{\pi}{3}.
\end{equation*}

From the above we conclude that each angle of the triangle $T = \conv \{x_1, x_2, x_3 \}$ belongs to the range $\left[\frac{\pi}{4}, \frac{3\pi}{4} \right]$.
A standard formula from elementary geometry $h_a = 2R \sin \beta \sin \gamma$ [ref:], where $R = r$ is the radius of the circumcircle of $T$
yields that each height of $T$ is at least $2r \sin^2 \frac{\pi}{4} = r > 2$. This contradicts the $T(B, 3)$ property of $X$.

\noindent{\bf Case 3.} $\# X = 5$. Let the points of $X$ be enumerated as $x_1, x_2, \ldots, x_5$ so that the polygonal line $x_1 x_2 \ldots x_5 x_1$ is the boundary of the convex pentagon
$\conv X$. Let, finally, $\phi_i$ ($i = 1, 2, \ldots, 5$) be half the central angular measure of the arc $x_i x_{i + 1}$ ($i_6 = i_1$) of $\partial (rB)$
that does not contain other points $x_j$. We consider two subcases.

\noindent{\bf Subcase 3.1.} $\max(\phi_1, \phi_2, \ldots, \phi_5) \geq \frac{\pi}{4}$. With no loss of generality assume that $\phi_1 \geq \frac{\pi}{4}$
Let $y_1, y_2 \in \partial (rB)$ be two points such that the polygonal line $x_1 x_2 y_2 y_1 x_1$ bounds a convex quadrangle, and the arcs $x_1y_1$
and $x_2y_2$ have central angular measure $\frac{\pi}{2}$ each. There are two smaller subcases.

\noindent{\bf Subcase 3.1.1.} The arc $y_1y_2$ (the one that does not contain $x_1$ and $x_2$) contains no points of $X$. Then $X$ is contained in the union of
arcs $x_1y_1$ and $x_2y_2$. Let us start moving all points of $X$ simultaneously over $\partial (rB)$ towards either $x_1$ or $x_2$, depending on which of the two
arcs the particular point belongs to, until John's condition degenerates. (Clearly, John's condition will degenerate before all points of $X$ arrive at either $x_1$ or $x_2$.)

At the moment John's condition degenerates, the modified set $X$ contains a 4-tuple $\{ z_1, z_2, z_3, z_4 \}$ satisfying the condition of Corollary~\ref{cor:four_point_identity}.
By Corollary~\ref{cor:no_large_arcs} one concludes that exactly two of the points $z_i$ (say, $z_1$ and $z_2$) belong to the arc $x_1y_1$, while the other two belong to the arc
$x_2y_2$. But this produces a contradiction, similarly to the proof of~\eqref{eq:max_at_least_quarter_pi}.

\noindent{\bf Subcase 3.1.2.} Some point $x_j$ belongs to the arc $y_1y_2$. Then each angle of the triangle $\conv \{ x_1, x_2, x_j \}$ is between $\frac{\pi}{4}$ and $\frac{3\pi}{4}$.
Similarly to Case 2, the property $T(B, 3)$ does not hold for $X$.

\noindent{\bf Subcase 3.2.} $\max(\phi_1, \phi_2, \ldots, \phi_5) < \frac{\pi}{4}$. With no loss of generality, let $\phi_1 = \max(\phi_1, \phi_2, \ldots, \phi_5)$. By
the Pigeonhole Principle, $\phi_1 \geq \frac{\pi}{5}$.

Consider the triangle $\conv \{ x_1, x_2, x_4 \}$. Let $\alpha$ be the angle at $x_2$. Then
\begin{equation*}
  \frac{\pi}{2} \geq \phi_4 + \phi_5 = \alpha = \pi - \phi_1 - \phi_2 - \phi_3 \geq \pi - 3 \phi_1.
\end{equation*}
Consequently, the height $h_{1,24}$ of the triangle from $x_1$ satisfies
\begin{equation*}
  h_{1, 24} = 2r \sin \phi_1 \sin \alpha \geq 2r \sin \phi_1 \sin 3\phi_1 > 2.
\end{equation*}
Similarly, $h_{2, 14} > 2$. Finally, $h_{4, 12} \geq 2r \sin^2 3\phi_1 > 2$. Again, $X$ violates the property $T(B, 3)$. \end{proof}

\section{Statements with computer-assisted proofs: the algorithms}\label{sec:computer_proofs}

The proofs of Lemma~\ref{lem:1.5} and parts (b)--(c) of Theorem~\ref{thm:grid} are computer-assisted. In this section we describe our approach towards the proof.

We claim that each of the statements has to be verified for a finite number of pairs $(r_1, r_2)$. Indeed, Lemma~\ref{lem:1.5} concerns a particular pair $(3, 1.62)$.
For Theorem~\ref{thm:grid} the case $r_2 < 1.62$ is immediate, while the case $r_2 > 2$ is impossible due to Lemma~\ref{lem:2}. Hence we are interested only in
the values $1.62 \leq r_2 \leq 2$. But after proving Lemma~\ref{lem:1.5}, we conclude that the case $r_1 > 3$ is impossible. Thus
\begin{equation*}
  (r_1, r_2) \in (0.15 \mathbb Z^2) \cap ([1.62, 3] \times [1.62, 2]),
\end{equation*}
which leaves us with a finite set of pairs to consider. Each pair $(r_1, r_2)$ is considered separately, therefore in the follow-up we assume that $r_1$ and $r_2$ are fixed.

Denote
\begin{equation*}
  I^p_n = \left[  \frac{2p \pi}{n}, \frac{2(p + 1) \pi}{n} \right], \qquad (p = 0, 1, \ldots, n - 1).
\end{equation*}
In order to prove each of the results, we use the so-called {\it divide-and-conquer} technique. The particular details are provided below.

\subsection{Lemma~\ref{lem:1.5}}
It is sufficient to consider the case $\# X = 5$ with the additional assumption that the set $NX$ satisfies~\eqref{eq:john_ball}, where $N$ is an affine map such that $NE = B$.
Indeed, the case $\# X > 5$ is ruled out by a small generic perturbation. In turn, for the case $\# X \leq 5$ there is a set $X' \subset \partial E$ such that $\# X' = 5$,
$NX'$ satisfies~\eqref{eq:john_ball} and each point of $X'$ is arbitrarily close to some point of $X$.

Denote
\begin{equation*}
  Q(p_1, p_2, \ldots , p_5; n) = I^{p_1}_n \times I^{p_2}_n \times \ldots \times I^{p_5}_n.
\end{equation*}
We start with $n = n_0 = 60$. Consider the set $\mathcal Q_0$ of all cubes $Q(p_1, p_2, \ldots , p_5; n_0)$ such that $0 \leq p_1 \leq p_2 \leq \ldots \leq p_5 < n_0$.

For each fixed $Q = Q(p_1, p_2, \ldots , p_5; n_0) \in \mathcal Q_0$ denote $\beta_i = \frac{2(p_i + 1/2) \pi}{n_0}$. Then one of the following holds:
\begin{enumerate}
  \item The set $\{ ( r_1 \cos \beta_i, r_2 \sin \beta_i) : i = 1, 2, \ldots, 5 \}$
        violates the property $T\left( \left(1 + r_1 \frac{\pi}{n} \right)B , 3 \right)$. Then, by Lemma~\ref{lem:lipschitz}, every 5-tuple $(\alpha_1, \alpha_2, \ldots, \alpha_5) \in Q$
        violates the property $T(B, 3)$.
  \item Consider the five values as in~\eqref{eq:5_functionals} with the arguments $\beta_1, \beta_2, \ldots, \beta_5$. If the largest of those values, $F_{\max}$ and the smallest, $F_{\min}$,
        satisfy
        \begin{equation*}
          F_{\max} > \frac{3\pi}{n}, \qquad F_{\min} < -\frac{3\pi}{n},
        \end{equation*}
        then no 5-tuple $\{ ( \cos \alpha_i, \sin \alpha_i) : i = 1, 2, \ldots, 5 \}$, where $(\alpha_1, \alpha_2, \ldots, \alpha_5) \in Q$ and $\alpha_1 < \alpha_2 < \ldots < \alpha_5$,
        satisfies~\eqref{eq:john_ball}.
  \item None of the above holds. Then we include all the cubes
        \begin{equation*}
          \{ Q(p'_1, p'_2, \ldots , p'_5; 2n_0) :  p'_i \in \{ 2p_i, 2p_i + 1 \}, p'_1 \leq p'_2 \leq \ldots \leq p'_5 \}
        \end{equation*}
        in the new set $\mathcal Q_1$.
\end{enumerate}
We apply the same procedure to the set $\mathcal Q_1$ and $n = n_1 = 2n_0$. Then we continue in the same fashion with $(\mathcal Q_2, n_2)$, etc. One obtains that $\mathcal Q_6 = \varnothing$,
which immediately implies Lemma~\ref{lem:1.5}.

\subsection{Theorem~\ref{thm:grid}, part (b)}
The setting of Theorem~\ref{thm:grid}, part (b) refers to 4-tuples of points. Therefore we consider the 4-dimensional cubes
\begin{equation*}
  Q(p_1, p_2, p_3, p_4; n) = I^{p_1}_n \times I^{p_2}_n \times \ldots \times I^{p_5}_n.
\end{equation*}
We start with $n = n_0 = 120$. Consider the set $\mathcal Q_0$ of all cubes $Q(p_1, p_2, p_3, p_4; n_0)$ such that $0 \leq p_1 \leq p_2 \leq p_3 \leq p_4 < n_0$.

For each fixed $Q = Q(p_1, p_2, p_3, p_4; n_0) \in \mathcal Q_0$ denote $\beta_i = \frac{2(p_i + 1/2) \pi}{n}$. Then one of the following holds:
\begin{enumerate}
  \item The set $Z = \{ ( r_1 \cos \beta_i, r_2 \sin \beta_i) : i = 1, 2, \ldots, 4 \}$
        violates the property $T\left( \left(1 + r_1 \frac{\pi}{n} \right)B , 3 \right)$. Then, by Lemma~\ref{lem:lipschitz}, every 5-tuple $(\alpha_1, \alpha_2, \ldots, \alpha_5) \in Q$
        violates the property $T(B, 3)$.
  \item $|F(\beta_1, \beta_2, \beta_3, \beta_4)| > \frac{3\pi}{n}$. Then every 4-tuple $Z' = \{ ( r_1\cos \alpha_i, r_2\sin \alpha_i) : i = 1, 2, 3, 4 \}$,
        where $(\alpha_1, \alpha_2, \alpha_3, \alpha_4) \in Q$ and $\alpha_1 < \alpha_2 < \alpha_3 < \alpha_4$, admits an elliptical disk $E' \supset Z'$ such that $|E'| < |E|$.
  \item $\beta_{i + 1} - \beta_i > \frac{2 \pi}{3}$ for some $i \in \{ 1, 2, 3, 4 \}$, where $\beta_5 = \beta_1 + 2\pi$. Then, with $\alpha_i$  and $Z'$ as above, the condition of
        Corollary~\ref{cor:no_large_arcs} is violated, thus $Z'$ admits an elliptical disk $E' \supset Z'$ such that $|E'| < |E|$. (Here we use that $n$ is a multiple of 3.)
  \item With $Z$ as above, the set $E \cap R\left( Z, r_1\frac{\pi}{n} \right)$ satisfies the property $T \left( \frac{\sqrt{5} + 1}{2} B \right)$, where
        \begin{equation*}
          R(Z, \varepsilon) = \{ x \in \mathbb R^2 : \text{$Z \cup \{x\}$ satisfies the property $T((1 + \varepsilon)B, 3)$} \}.
        \end{equation*}
        Then $R(Z') \subseteq R(Z, \varepsilon)$, hence the conclusion of Theorem~\ref{thm:grid}, part (b) holds whenever
        $(\alpha_1, \alpha_2, \alpha_3, \alpha_4) \in Q$ and $\alpha_1 < \alpha_2 < \alpha_3 < \alpha_4$.
  \item None of the above holds. Then we include all the cubes
        \begin{equation*}
          \{ Q(p'_1, p'_2, p'_3, p'_4; 2n_0) :  p'_i \in \{ 2p_i, 2p_i + 1 \}, p'_1 \leq p'_2 \leq p'_3 \leq p'_4 \}
        \end{equation*}
        in the new set $\mathcal Q_1$.
\end{enumerate}
We apply the same procedure to the set $\mathcal Q_1$ and $n = n_1 = 2n_0$. Then we continue in the same fashion with $(\mathcal Q_2, n_2)$, etc. One obtains that $\mathcal Q_5 = \varnothing$,
which immediately implies Theorem~\ref{thm:grid}, part (b).

\subsection{Theorem~\ref{thm:grid}, part (c)}
The algorithm repeats the one from the previous subsection with the following changes.

\begin{enumerate}
  \item Five-dimensional cubes are used instead of four-dimensional, since this part of Theorem~\ref{thm:grid} refers to 5-tuples of points.
  \item The condition for the minimality of $E$ is treated similarly to Lemma~\ref{lem:1.5}.
  \item Having obtained the pair $(\mathcal Q_4, n_4)$, we observe that $n_4 = 1920$. Therefore for every each cube $Q \in \mathcal Q_4$ it is sufficient to check its center.
        This is accomplished straightforwardly.
\end{enumerate}

\section{Proof of Theorem~\ref{thm:grid}, part (a)}\label{sec:grid_a}

If $r_2 \leq \frac{\sqrt{5} + 1}{2}$, then $E$ satisfies the property $T\left( \frac{\sqrt{5} + 1}{2} B \right)$. We will show that the case $r_2 > \frac{\sqrt{5} + 1}{2}$
is impossible. Namely, since $\# Z = 3$, it will be sufficient to show that $Z$ violates the property $T(B)$.

Let $N$ be an affine map such that $NE = \frac{\sqrt{5} + 1}{2} B$. $N$ satisfies~\eqref{eq:contraction}, hence it is sufficient to prove that $NZ$ violates the property $T(B)$.
But this is clear, because $\conv NZ$ is a regular triangle of height $\frac{3}{2} \cdot \frac{\sqrt{5} + 1}{2} > 2$.

\section{Reduction of Theorem~\ref{thm:estimate} to Theorem~\ref{thm:grid}}\label{sec:proof_est}

Now, assuming that Theorem~\ref{thm:grid} is verified, we turn to the proof of Theorem~\ref{thm:estimate}.

\begin{proof}[Proof of Theorem~\ref{thm:estimate}]
We argue by contradiction. Assume that Theorem~\ref{thm:estimate} is false. Then there exists a set $X_0$ and a constant $\varepsilon > 0$
such that $X$ satisfies the property $T(B, 3)$, but does not satisfy the property $T((c + \varepsilon)B)$. Then the set
$X_1 = \frac{c + \varepsilon / 2}{c + \varepsilon} \cdot X_0$ satisfies the property $T\left(\frac{c + \varepsilon / 2}{c + \varepsilon} \cdot B, 3 \right)$,
but not the property $T((c + \varepsilon / 2)B)$. Finally, every sufficiently small perturbation $X$ of the set $X_1$
satisfies the property $T(B, 3)$, but does not satisfy the property $T(cB)$. As in the previous section, one can choose $X$ to be sufficiently generic to
guarantee that no ellipse passes through six different points of $X$.

Let $E$ be the elliptical disk of minimal area containing $X$. Consider an arbitrary affine map $N$ such that $NE = B$. Then there exists a finite subset
$Z = \{ z_1, z_2, \ldots, z_k \} \subseteq X \cap \partial E$ and positive coefficients $c_1, c_2, \ldots, c_k$ such that the identities~\eqref{eq:john_ball}
hold for $x_i = Nz_i$. Of course, $k \leq 5$, because we assume $X$ to be generic. On the other hand, $k \geq 3$, since~\eqref{eq:john_ball} cannot
hold for $k = 1, 2$. As in Conjecture~\ref{conj:finite}, we use the notation $R(Z)$ defined by~\eqref{eq:r}.
Now we proceed by case analysis.

\noindent{\bf Case 1.} $k = 3$. Let $Z = X \cap \partial E$. Using case (a) of Theorem~\ref{thm:grid}, we conclude that the set $E \cap R(Z)$ satisfies the property $T\left( \frac{\sqrt{5} + 1}{2} \cdot B \right)$.
But $X \subseteq E \cap R(Z)$. Therefore the set $X$ satisfies the property $T\left( \frac{\sqrt{5} + 1}{2} \cdot B \right)$, too.

\noindent{\bf Case 2.} $k = 4$. Consider two subcases.

\noindent{\bf Subcase 2.1.} $r_2 \leq 1.62$. In this subcase it is clear that $E$ satisfies the property $T(1.62 B)$. Since $X \subset E$, the set $X$ satisfies the property $T(1.62 B)$, too.

\noindent{\bf Subcase 2.2.} $r_2 > 1.62$. Let $r'_1, r'_2 \in 0.015 \mathbb Z$ satisfy
\begin{equation*}
  r'_1 \leq r_1 < r'_1 + 0.015, \qquad r'_2 \leq r_2 < r'_2 + 0.015.
\end{equation*}
Consider the following affine maps $M_1$ and $M_2$:
\begin{equation*}
(x, y) \xmapsto{M_1} \left( \frac{r'_1}{r_1}x, \frac{r'_2}{r_2}y \right), \quad (x, y) \xmapsto{M_2} \left( \frac{r'_2 + 0.015}{r'_2}x, \frac{r'_2 + 0.015}{r'_2}y \right).
\end{equation*}
Consider an arbitrary triple $\{x_1, x_2, x_3 \} \subseteq X$. By condition of the lemma, it satisfies the property $T(B)$. Since $M_1$ satisfies~\eqref{eq:contraction},
the triple $\{M_1 x_1, M_1 x_2, M_1 x_3 \}$ satisfies the property $T(B)$, too. Therefore the set $M_1X$ satisfies the property $T(B, 3)$. Hence $M_1X \subset M_1 E \cap R(M_1 Z)$.

By Theorem~\ref{thm:grid}, case (b), the set $M_1 E \cap R(M_1 Z)$ satisfies the property $T\left( \frac{\sqrt{5} + 1}{2} \cdot B \right)$. Therefore the set $M_1X$ satisfies the property
$T\left( \frac{\sqrt{5} + 1}{2} \cdot B \right)$.

Consequently, the set $(M_2 M_1)X$ satisfies the property $T\left(\frac{r'_2 + 0.015}{r'_2} \cdot \frac{\sqrt{5} + 1}{2} \cdot B \right)$. Since
\begin{equation*}
  \frac{r'_2 + 0.015}{r'_2} \cdot \frac{\sqrt{5} + 1}{2} \leq \frac{1.635}{1.62} \cdot \frac{\sqrt{5} + 1}{2} < 1.635,
\end{equation*}
we conclude that $(M_2 M_1)X$ satisfies the property $T(1.635 B)$.

Finally, the map $(M_2 M_1)^{-1}$ satisfies~\eqref{eq:contraction}. Hence the set $X$ satisfies the property $T(1.635 B)$.

\noindent{\bf Case 3.} $k = 5$. Consider three subcases.

\noindent{\bf Subcase 3.1.} $r_2 \leq 1.645$. By definition of the subcase, $E$ satisfies the property $T(1.645 B)$. But $X \subset E$, hence the set $X$ satisfies the property $T(1.645 B)$, too.

\noindent{\bf Subcase 3.2.} $r_2 > 1.645$, $r_1 > 3$. Consider the map $M$ defined by
\begin{equation*}
  (x, y) \xmapsto{M} \left( \frac{3}{r_1}x, \frac{1.62}{r_2}y \right).
\end{equation*}
The map $M$ satisfies~\eqref{eq:contraction}, and the set $X$ satisfies the property $T(3, B)$. Therefore, by Lemma~\ref{lem:contraction}, $MX$ satisfies the property $T(3, B)$ as well.
But this is a contradiction to Lemma~\ref{lem:1.5}, hence the subcase is impossible.

\noindent{\bf Subcase 3.3.} $1.645 < r_2 \leq r_1 \leq 3$. We need some further notation. Define the constants $r'_1, r'_2 \in 0.015 \mathbb Z$ by the inequalities
\begin{equation*}
  r'_1 \leq 0.995 r_1 < r'_1 + 0.015, \qquad r'_2 \leq 0.995 r_2 < r'_2 + 0.015.
\end{equation*}
We introduce the maps $M_1$ and $M_2$ defined by
\begin{equation*}
(x, y) \xmapsto{M_1} \left( \frac{r'_1}{r_1}x, \frac{r'_2}{r_2}y \right), \quad (x, y) \xmapsto{M_2} \left( \frac{r'_2 + 0.015}{0.995 r'_2}x, \frac{r'_2 + 0.015}{0.995 r'_2}y \right).
\end{equation*}
If $N$ is the map defined by
\begin{equation*}
  (x, y) \xmapsto{N} \left( \frac{1}{r_1}x, \frac{1}{r_2}y \right),
\end{equation*}
then the identity $NE = B$ holds. Then, by definition of the subcase, there is a subset $Z = \{ z_1, z_2, z_3, z_4, z_5 \} \in X \cap \partial E$ such that
the points $x_i = Nz_i$ satisfy~\eqref{eq:john_ball}. The points $x_i$ can be parameterized by the parameters $\alpha_i$ so
that $x_i = (\cos \alpha_i, \sin \alpha_i)$. With no loss of generality, assume that
\begin{equation*}
  0 \leq \alpha_1 < \alpha_2 < \ldots < \alpha_5 < 2\pi.
\end{equation*}
Define $\alpha'_i \in \frac{\pi}{960} (\mathbb Z + 1/2)$ ($i = 1, 2, \ldots, 5$) by the inequalities
\begin{equation*}
  \alpha_i - \frac{\pi}{1920} \leq \alpha'_i < \alpha_i + \frac{\pi}{1920}.
\end{equation*}
Finally, for every $t \in [0, 1]$ let
\begin{equation*}
  \alpha_i(t) = (1 - t)\alpha_i + t \alpha'_i, \qquad x_i(t) = (\cos \alpha_i(t), \sin \alpha_i(t)).
\end{equation*}
Now consider two subcases.

\noindent{\bf Subcase 3.3.1.} For every $t \in [0, 1]$ there exist positive coefficients $c_i(t)$ ($i = 1, 2, \ldots, 5$) such that substitution of $x_i(t)$ instead of $x_i$
and $c_i(t)$ instead of $c_i$ turns the identities~\eqref{eq:john_ball} into correct ones. By definition of $r'_1$ and $r'_2$,
one has $\max \left( \frac{r'_1}{r_1}, \frac{r'_2}{r_2} \right) \leq 0.995$. Therefore the set $M_1X$ satisfies the property $T(0.995B, 3)$. Next,
\begin{equation*}
  \| M_1 z_i - M_1(N^{-1} x_i(1) \| \leq r'_1 (\alpha_i - \alpha'_i) < \frac{3\pi}{1920} < 0.005.
\end{equation*}
Thus, by Lemma~\ref{lem:lipschitz}, the set $Y = M_1 (X \cup \{ N^{-1} x_1(1), N^{-1} x_2(1), \ldots, N^{-1} x_5(1) \})$ satisfies the property $T(B, 3)$.
But $M_1N^{-1} x_1(1) \in Y$, therefore the case (c) of Theorem~\ref{thm:grid} is applicable. Hence the set $M_1 X \supseteq Y$ satisfies the property $T\left( \frac{\sqrt{5} + 1}{2} \cdot B \right)$.

Consequently, the set $(M_2 M_1)X$ satisfies the property $T\left(\frac{r'_2 + 0.015}{0.995 r'_2} \cdot \frac{\sqrt{5} + 1}{2} \cdot B \right)$. Since
\begin{equation*}
  \frac{r'_2 + 0.015}{0.095 r'_2} \cdot \frac{\sqrt{5} + 1}{2} \leq \frac{1.635}{0.995 \cdot 1.62} \cdot \frac{\sqrt{5} + 1}{2} < 1.645,
\end{equation*}
we conclude that $(M_2 M_1)X$ satisfies the property $T(1.645 B)$.

Finally, the map $(M_2 M_1)^{-1}$ satisfies~\eqref{eq:contraction}. Hence the set $X$ satisfies the property $T(1.645 B)$.

\noindent{\bf Subcase 3.3.2.} The condition of Subcase 3.3.1 does not hold for $t \in T$, where $T \in (0, 1]$ is a non-empty set. (By definition, $x_i(0) = x_i$, hence
the condition necessarily holds for $t = 0$.) Corollary~\ref{cor:simplex} immediately implies that for $t_0 = \inf T$ there is a proper subset
$X_0 \subsetneq \{x_1(t_0), x_2(t_0), \ldots, x_5(t_0) \}$ such that $B$ is the minimum area elliptical disk containing $X_0$.

Similarly to Subcase 3.3.1, the set $Y = M_1 (X \cup N^{-1} X_0)$ satisfies the property $T(B, 3)$. But we have either $\# X_0 = 3$ or $\# X_0 = 4$,
therefore either case (a) or case (b) of Theorem~\ref{thm:grid} is applicable. Hence the set $M_1 X \supseteq Y$ satisfies the property $T\left( \frac{\sqrt{5} + 1}{2} \cdot B \right)$.
The rest of the argument proceeds exactly as in Subcase 3.3.1.
\end{proof}

\end{document}